\documentclass[english, 12pt]{amsart}
\usepackage{amssymb}
\usepackage{amsfonts}
\usepackage{amscd}
\usepackage[T1]{fontenc}
\usepackage{color}
\usepackage[utf8]{inputenc}

\definecolor{immi}{rgb}{0,.6,.1}

\long\def\change#1{{#1}}

\newbox\removebox
\newcommand\remove[2]{%
\setbox\removebox=\ifmmode\hbox{$#2$}\else\hbox{#2}\fi%
\leavevmode
\rlap{\textcolor{#1}{\vrule height0.8ex depth-0.6ex width\wd\removebox}}%
\box\removebox
}
\long\def\bigremove#1{%
\par\setbox\removebox=\vbox{#1}%
\vbox{%
\vbox to0pt{\hbox{\tikz\draw[color=blue,thick] (0,0) -- (\wd\removebox,-\ht\removebox)  (\wd\removebox,0) -- (0,-\ht\removebox);}}
\box\removebox
}
}

\usepackage{mathrsfs} 

\makeatletter
\def\RFss@@#1{\RF^*_{\!*#1}}
\def\RFss@_#1{\RFss@@{,#1}}
\def\RFss{\@ifnextchar_{\RFss@}{\RFss@@{}}}
\makeatother

\newcommand{\RF}{{\rm RF}}

\def\Supp{\operatorname{Supp}}

\def\lct{\operatorname{lct}}
\def\moi{\operatorname{moi}}

\def\ac{{\overline{\rm ac}}}
\def\acc{{{\rm ac}}}
\def\Supp{\operatorname{Supp}}

\def\11{{\mathbf 1}}
\def\AA{{\mathbb A}}

\def\CC{{\mathbb C}}

\def\FF{{\mathbb F}}

\def\NN{{\mathbb N}}

\def\PP{{\mathbb P}}
\def\QQ{{\mathbb Q}}

\def\ZZ{{\mathbb Z}}

\def\cM{{\mathcal M}}

\def\cO{{\mathcal O}}

\def\cZ{{\mathcal Z}}

\newcommand{\Q}{{\mathbb Q}}

\newcommand{\llparenthesis}{(\negthinspace(}
\newcommand{\rrparenthesis}{)\negthinspace)}

\newtheorem{thm}{Theorem}[section]
\newtheorem{lem}[thm]{Lemma}
\newtheorem{cor}[thm]{Corollary}
\newtheorem{prop}[thm]{Proposition}

\theoremstyle{definition}
\newtheorem{defn}[thm]{Definition}

\newtheorem{def-prop}[thm]{Proposition-Definition}
\newtheorem{def-theorem}[thm]{Theorem-Definition}
\newtheorem{def-lem}[thm]{Lemma-Definition}

\theoremstyle{remark}
\newtheorem{remark}[thm]{Remark}

\theoremstyle{plain}

\numberwithin{equation}{section}

  {\par\medskip\noindent #1\par\begingroup%
    \advance\leftskip by 1em\advance\rightskip by 1em}%
  {\par\endgroup}

\DeclareMathOperator*{\Spec}{Spec}

\newcommand{\ord}{\operatorname{ord}}

\def\cO{\mathcal{O}}

\def\frm{\mathfrak{m}}
\def\frq{\mathfrak{q}}

\renewcommand{\phi}{\varphi}
\renewcommand{\epsilon}{\varepsilon}
\renewcommand{\theta}{\vartheta}

\def\Q{{\mathbf Q}}

\renewcommand{\and}{ \quad \text{and} \quad }

\begin{document}

\setcounter{tocdepth}{1} 

\author[R.~Cluckers]
{Raf Cluckers}
\address{Universit\'e de Lille, Laboratoire Painlev\'e, CNRS - UMR 8524, Cit\'e Scientifique, 59655
Villeneuve d'Ascq Cedex, France, and,
KU Leuven, Department of Mathematics,
Celestijnenlaan 200B, B-3001 Leu\-ven, Bel\-gium}
\email{Raf.Cluckers@univ-lille.fr}
\urladdr{http://rcluckers.perso.math.cnrs.fr/}

\author[M.~Musta{\c{t}}{\v{a}}]{Mircea Musta{\c{t}}{\v{a}}}
\address{Department of Mathematics, University of Michigan, Ann Arbor, MI 48109, USA}
\email{mmustata@umich.edu}

\author[K.~H.~Nguyen]
{Kien Huu Nguyen}
\address{KU Leuven, Department of Mathematics,
Celestijnenlaan 200B, B-3001 Leu\-ven, Bel\-gium}
\email{kien.nguyenhuu@kuleuven.be}

\thanks{The authors R.C. and K.H.N. are partially supported by the European Research Council under the European Community's Seventh Framework Programme (FP7/2007-2013) with ERC Grant Agreement nr. 615722
MOTMELSUM, by the Labex CEMPI  (ANR-11-LABX-0007-01), and by KU Leuven IF C14/17/083. \change{K.H.N. is partially supported by Fund for Scientific Research - Flanders (Belgium) (F.W.O.) 12X3519N.} M.M. is partially supported by NSF grant DMS-1701622.
The authors would like to thank \change{Ben Lichtin}, Andrei Musta\c{t}\u{a}, Johannes Nicaise, Antonio Rojas-Le\'on, and Wim Veys for advice and interesting discussions on the topics of this paper and \change{thank the referee for valuable comments. The authors are grateful} to Jan Denef for his long term guidance in this subject.}



\begin{abstract}
We prove an upper bound on the log canonical threshold of a hypersurface that satisfies a certain power condition
and use it to prove several generalizations of Igusa's conjecture on exponential sums, with the log-canonical threshold in the exponent of the estimates.  We show that this covers optimally all situations of the conjectures for non-rational singularities, by comparing the log canonical threshold with a local notion of the motivic oscillation index.
\end{abstract}
\title[Igusa's conjecture for exponential sums ]
{Igusa's conjecture for exponential sums: optimal estimates for non-rational singularities}

\maketitle

\section{Introduction}

\subsection*{} Igusa's conjecture on exponential sums predicts upper bounds for $|S_f(a)|$
in terms of $a$, where $f$ is a nonconstant polynomial over $\ZZ$ in $n$ variables, $a$ runs over the positive integers, and $S_f(a)$ is the finite exponential sum
\begin{equation}\label{eq:Sfa}
S_{f}(a):= \frac{1}{a^n} \sum_{x\in (\ZZ/a\ZZ)^n} \exp \left(\frac{2\pi i f(x) }{ a }\right).
\end{equation}
If $a$ runs only over the positive powers $p^m$ of a fixed prime number $p$, these bounds are well-known and proved by Igusa; the key point of his conjecture is about varying the prime $p$, as follows. Suppose that for some real $\sigma>0$ and for each prime $p$ there exists a constant $c_p>0$ such that 
\begin{equation}\label{eq:sigma}
|S_{f}(p^m)|  <c_p p^{-m\sigma} \mbox{ for each integer $m\geq 2$}.
\end{equation}
Then the question 
is generally whether one can take $c_p$ independently of the prime $p$ (but still depending on $\sigma$ and $f$). In a more explicit form, Igusa puts forward precise values for $\sigma$, the infimum of which relates to the log-canonical threshold of $f$ in the case of non-rational singularities, and, more generally, to the motivic oscillation index (see below). Conditions on $f$ that were originally imposed by Igusa \cite[p.~2 and 170]{Igusa3} (like the homogeneity of $f$ and bounds on $\sigma$) have been relaxed in several later variants of his question \cite{DenSper, CDenSperlocal, CVeys, CAN}.

In this paper we prove these later variants with the log-canonical threshold playing a key role both in the exponent of the upper bounds and in the proofs. The appearance of the log-canonical threshold in the exponent of the upper bounds is optimal in many cases: it is only when the hypersurfaces defined by $f-c$, for $c\in {\mathbb C}$, have at worst rational singularities, that there is room for an improved exponent, see Section \ref{sec:opt}.

We derive the bounds for the exponential sums from an upper bound on the log-canonical threshold of the hypersurface defined by $f$
in the presence of a certain \emph{power condition}. This is described in terms of a log resolution for the hypersurface.
We prove this upper bound by making use of a finiteness result concerning certain divisorial valuations, result which follows from the recent progress in the Minimal Model Program
\cite{BCHM}.
Deriving Igusa's conjecture from the log-canonical threshold bound relies on
several subtle results on Igusa's local zeta functions, most of which can be found in the overview paper \cite{DenefBour} by Denef. These allow us to reduce to finite field exponential sums with multiplicative characters that can be bounded in a way matching our bounds on the log-canonical threshold.

We mention that while Igusa's conjectured upper bounds are very natural, his motivation came from
their role in obtaining ad\`{e}lic integrability properties, which in turn were crucial for proving
Poisson summation formulae throughout his work \cite{Igu-Invent, Igu-Pfaff, Igu-Poiss1, Igusa2, Igu-exp, Igu-Poiss2, Igusa3}. \change{These Poisson summation formulas are a step in  Igusa's program towards new local-global principles, see \cite[p.~240]{Igu-exp}.}

\subsection{}\label{subsec:bound}

We begin by stating our result on log canonical thresholds.
 Let $X$ be a smooth complex algebraic variety and $D$ a non-empty hypersurface in $X$ defined by
$f\in \cO_X(X)$. We consider a log resolution  $h\colon Y\to X$ of the pair $(X,D)$,
which is an isomorphism over $X\smallsetminus D$. Therefore
$h$ is a projective morphism, $Y$ is smooth, and $h^*(D)$ is a divisor with simple normal crossings. Note that by our assumption on $h$,
the relative canonical divisor $K_{Y/X}$ is supported on $h^*(D)_{\rm red}$.

We write
$$h^*(D)=\sum_{i=1}^N N_i E_i\quad\text{and}\quad K_{Y/X}=\sum_{i=1}^N (\nu_i-1)E_i,$$
for positive integers $N_i$, $\nu_i$, \change{and prime divisors $E_{i}$},
so that the log canonical threshold $\lct(f)$ of $f$ is given by
$$\lct (f)=\min_i\frac{\nu_i}{N_i}$$
(see, for example, \cite{Mustata2} for an introduction to log canonical thresholds).
For every subset $I\subseteq\{1,\ldots,N\}$, we put
$$E_I=\bigcap_{i\in I}E_i\quad\text{and}\quad E_I^{\circ}=E_I\smallsetminus\bigcup_{i\not\in I}E_i,$$
with $E_\emptyset=Y$.
By construction, given a (possibly non-closed) point $P\in Y$, there is a unique $I\subseteq \{1,\ldots,N\}$ such that
$P\in E_I^{\circ}$; moreover, there is an algebraic system of coordinates $x_1,\ldots,x_n$ in a neighborhood $X_0$ of $P$ such that,
after relabeling so that $I=\{1,\ldots,m\}$, we have $E_i\cap X_0=V(x_i)$ if and only if $i\leq m$ and we can write
\begin{equation}\label{eqn1_bound_lct}
f\circ h\vert_{X_0}=u\cdot\prod_{i=1}^mx_i^{N_i},
\end{equation}
where $u\in\cO_Y(X_0)$ is an invertible regular function on $X_0$.

For a closed subset (or closed subscheme) $Z$ of $X$, having non-empty intersection with $D$, we denote by $\lct _Z(f)$
the largest log canonical threshold $\lct (f\vert_V)$, where $V$ is an open neighborhood of $Z$.
Note that we have $\lct_Z(f)=\min_{x\in Z\cap D}\lct_x(f)$.

\begin{thm}\label{thm_bound}
Suppose that there is a non-empty
subset $I\subseteq \{1,\ldots,N\}$ and an open subset $X_0$ as above such that $X_0\cap E_I^{\circ}$ is non-empty
and
\begin{equation}\label{eq:power}
u\vert_{X_0\cap E_I^{\circ}}=g^d
\end{equation}
for some integer $d>1$ with $d\vert N_i$ for all $i\in I$ and some $g\in\cO(X_0\cap E_I^{\circ})$.
In this case, we have the following upper bound
for the log canonical threshold of $f$:
\begin{equation}\label{eq_thm}
\lct (f)\leq \frac{1}{d}+\sum_{i\in I}N_i\left(\frac{\nu_i}{N_i}-\lct (f)\right).
\end{equation}
More generally, if
$\overline{h(X_0\cap E_I^{\circ})}\cap Z\neq\emptyset$, then we have
\begin{equation}\label{eq_thm_Z}
\lct _Z(f)\leq \frac{1}{d}+\sum_{i\in I}N_i\left(\frac{\nu_i}{N_i}-\lct _Z(f)\right).
\end{equation}
\end{thm}

\begin{remark}
\change{A slightly more restrictive form of the condition} in (\ref{eq:power}) of Theorem \ref{thm_bound} is formalized and coined a ``power condition"  in Section \ref{sec:resolK} below. 
Note that the requirement $d\vert N_i$ for all $i\in I$ in the theorem, makes the condition (\ref{eq:power})
 independent of the choice of local coordinates $x_1,\ldots,x_n$.
\end{remark}

\subsection{}
Let us now formulate our main results on exponential sums. We work over a ring of integers $\cO$, instead of over $\ZZ$.

Let $f$ be a non-constant polynomial in $\cO[x]$ in $n$ variables $x=(x_1,\ldots,x_n)$.
Let $Z$ be a closed subscheme of $\AA_\cO^n$.
Let $L$ be a local field over the ring $\cO$, that is, a finite field extension of $\QQ_p$ or of $\FF_p\llparenthesis t\rrparenthesis$
for some prime number $p$ such that moreover there is a unit-preserving ring homomorphism $\cO\to L$.
We denote by $\cO_L$  the valuation ring of $L$ and by $|dx|$ the Haar measure on $L^n$, normalized so that $\cO_L^n$ has measure $1$. The number of elements in the residue field $k_L$ of $L$ is $q_L$ and equals a power of a prime number $p_L$.
Let $\psi\colon L\to\CC^\times$ be a nontrivial additive character on $L$, that is, a (nontrivial) continuous group homomorphism from the additive group of $L$ to $\CC^\times$.
For such data, consider the integral
\begin{equation}\label{eq:EfZ}
E_{f,L,\psi}^{Z} := \int_{\{x\in \cO_{L}^n\mid \overline x\in Z(k_L) \} }  \psi(f(x))|dx| ,
\end{equation}
where $\overline x$ stands for the image of $x$ under the natural projection $\cO_{L}^n \to k_L^n$. We write
$
E_{f,L,\psi}
$
for $E_{f,L,\psi}^{Z}$ when $Z= \AA_\cO^n$.
Note that the integrals $E_{f,L,\psi}^{Z}$ are in fact finite exponential sums which include the above sums $S_{f}(p^m)$, with $p$ a prime number, as special cases. Moreover, estimating $S_{f}(p^m)$ is the key to estimating $S_f(a)$ for general $a$, by the Chinese remainder theorem. In what follows, we write $\lct_Z(f)$ for $\lct_{Z_{\CC}}(f)$.

\begin{defn}\label{defn:mpsi}
For a nontrivial additive character $\psi$ on $L$, let $m_{\psi}$ be the unique integer $m$ such that $\psi$ is trivial on $\varpi_L^{m} \cO_L$ and nontrivial on $\varpi_L^{m-1}\cO_L$, where $\varpi_L$ is a uniformizer of $\cO_L$.
\end{defn}

Note that for any additive character $\psi$ on $L$ with $m_\psi=0$ and any $z\in L^\times$, the character $\psi_z$ sending $x\in L$ to $\psi(zx)$ satisfies
$$
m_{\psi_z} = - \ord (z),
$$
and that all nontrivial additive characters on $L$ are of the form $\psi_z$ for varying $z$, see e.g.~\cite{Igusa:intro}.

\begin{defn}\label{defn:sigma0}
For a non-constant polynomial $f$ with coefficients in $\CC$ and any subset $Z$ of $\CC^n$, let
$$
\sigma_Z(f)=\min\{\lct_x(f-b)\mid x\in Z,\ b=f(x)\}.
$$
If $Z=\CC^n$, then we simply write $\sigma(f)$ for $\sigma_Z(f)$. Given a closed subscheme $Y$ of $\AA^n_\cO$ or of $\AA^n_\CC$, we
write $\sigma_Y(f)$ for $\sigma_{Y(\CC)}(f)$.
\end{defn}

\begin{thm}[Exponential sums around $Z$]
\label{thm:sum}
\label{CVlct}
If $f\in\cO[x_1,\ldots,x_n]$ is a non-constant polynomial, and $Z$ is any closed subscheme of $\AA_\cO^n$, then
there exist $c>0$ and $M>0$ such that
\begin{equation}\label{eq:global}
|E^Z_{f,L,\psi}| <  c m_{\psi}^{n-1} q_L^{- \sigma_Z(f) m_{\psi} }
\end{equation}
for all local fields $L$ over $\cO$ whose residue field characteristic is at least $M$ and for all nontrivial additive characters $\psi$ on $L$ satisfying $m_{\psi}\geq 2$.
\end{thm}

The variant of Theorem \ref{thm:sum} with $Z=\{0\}$ is the Denef-Sperber conjecture from \cite{DenSper}, with the log-canonical threshold in the exponent.
Theorem \ref{CVlct}  with $Z=\AA_\cO^n$ covers the variant \cite[Conjecture 1.2 (1.2.1)]{CVeys} of Igusa's conjecture. 
At the end of Section \ref{igusa}, we will state and prove a version of Theorem \ref{CVlct} which is moreover uniform in the choice of $Z$, thus solving and generalizing the complete Conjecture 1.2 from \cite{CVeys}.
We moreover show the optimality of these estimates in the case of non-rational singularities, by providing lower bounds in Section \ref{sec:opt}, where we also formulate the remaining part of Igusa's conjecture with precise and optimal estimates with the so-called motivic oscillation index of $f$ around $Z$ in the exponent.
We also give an application of Theorem \ref{thm_bound} about poles of maximal possible order of Igusa's local zeta functions in the twisted case, see Section \ref{subsec:poles}.

\subsection{Remarks on
Igusa's conjecture}\label{sec:variants}

Igusa \cite[p.~2 and 170]{Igusa3} originally imposed two extra conditions in his conjecture: he focused on the case of homogeneous $f$ (mainly because in that case
 $0$ is the unique critical value), and he assumed  $\sigma>1$ in (\ref{eq:sigma}), since he wanted to derive ad\`{e}lic $L^1$-integrability of an ad\`elic variant of $E_{f,L,\psi}$ from his conjectural bounds on  exponential sums. In Igusa's situation, lower values of $\sigma$ in (\ref{eq:sigma}) yield only ad\`{e}lic $L^q$-integrability for higher $q$.

The exponent
$\sigma_Z(f)$
in Theorem \ref{CVlct} is not always optimal. However, it is optimal in the case that the hypersurface given by $f-b$ for some $b\in\CC$ has some non-rational singularities on each open $V$ containing $Z_\CC$.  We show this optimality in two steps: we introduce the notion of motivic oscillation index of $f$ around $Z$, denoted by $\moi_Z(f)$  (as a variant of a notion of \cite{Cigumodp}, in particular with a different sign), we show lower bounds of $|E^Z_{f,L,\psi}|$ with   $\moi_Z(f)$ in the exponent instead of $\sigma_Z(f)$, and we compare the values of $\moi_Z(f)$ with $\sigma_Z(f)$ in the case of non-rational singularities.
It may be interesting to study a relation between the motivic oscillation index and the  notion of
minimal exponents introduced in \cite[p.~52]{Saito-rational}. In the non-homogeneous case, the case $m_\psi=1$ can be problematic if one uses the motivic oscillation index in the exponent, as witnessed by $f(x,y)= x^2y-x$, see Example 7.2 of \cite{Cigumodp}. However, the case $m_\psi=1$ still makes sense by \cite{Cigumodp} for weighted homogeneous $f$ (and even possibly more generally, see the discussion following Conjecture 1.2.2 in \cite{CAN}).
For general $f$, it is more sensible to restrict to $m_\psi\geq 2$, as observed in \cite{CVeys}, where moreover the case $m_\psi\leq 4$ is proved.

Let us now discuss some previously proved cases of Igusa's conjecture on exponential sums, and its variants.
The case of the above theorems where $\sigma_Z(f)\leq 1/2$  has been recently obtained in \cite{Saskia-Kien}. Igusa treated (optimally) the case of homogeneous polynomials $f$ having an isolated singularity at $0$, see
 \cite[Section 5.3]{Igusa3}. 
For polynomials $f$ that are non-degenerate with respect to their Newton polyhedron at the origin, the exponential sums $E_{f,L,\psi}$ and $E^{\{0\}}_{f,L,\psi}$ are well understood,  see \cite{DenSper}, \cite{CDenSper}, \cite{CDenSperlocal}, and \cite{CAN}. Moreover, in the non-degenerate case, it is expected that the bounds from \cite{DenSper, CDenSper, CDenSperlocal, CAN} are optimal, see the questions about optimality and a certain vertex condition from \cite[Thms 5.17 - 5.19]{DenHoor} and \cite{Hoorn}.
For homogeneous polynomials in 3 variables and for weighted homogeneous polynomials in 2 variables, see
\cite{Lichtin4, Lichtin2, JWright}. \change{It is most likely that Lichtin's method based on good wedge decomposition from \cite{Lichtin4, Lichtin2}, in combination with Corollary \ref{cor:num} and with Cochrane's bounds from \cite[Theorem 1]{Cochrane}, can also be used to yield Theorem \ref{thm:sum}, and, a similar remark holds for the method of \cite{Saskia-Kien} based on arc spaces; both these approaches can probably avoid the use of multiplicative characters.}

\section{Proof of Theorem \ref{thm_bound}}\label{sec:bound}

\subsection{Log resolutions}\label{sec:resolK}
We first fix some terminology for log resolutions, based on  \cite{Hir:Res}. Let $K$ be a subfield of $\CC$ and $X_K$ a smooth,
geometrically connected variety over $K$. In the applications to Igusa's conjectures, we will take $X_K={\mathbb A}_K^n$, but in this section it is convenient to set up the notation in a more general setting. For a field extension $K'$ of $K$, we put $X_{K'}=X_K\times_{\Spec(K)}\Spec(K')$.
Given a nonzero, non-invertible \change{$f\in \cO_{X_K}(X_K)$}, we denote by $D_K$ the closed subscheme of $X_K$ defined by $f$ and put $D=D_{\CC}$.
Let $h\colon Y\to X_K$ be a log resolution of the pair $(X_K,D_K)$ (the existence of such $h$ follows from \cite[page 142, Main Theorem II]{Hir:Res}).

More precisely, we have:
 \begin{itemize}
\item $Y$ is a smooth closed subscheme of $\PP_{X_K}^{k}$, for some $k\geq 0$.
\item $h$ is a proper birational morphism which is an isomorphism over the complement of the support of $D_K$.
\item The divisor  $h^*(D_K)$ on $Y$ equals $\sum_{j\in T_K} N_{j}E_{j}$ for a finite set $T_K$  and some positive integers $N_j$, where each $E_{j}$ is
a prime divisor.
\item The divisor $h^*(D_K)$ has simple normal crossings, that is, if $I\subseteq T_K$ is such that $i\in I$ if and only if $a\in E_i$
and if we write in some neighborhood $V$ of $a$
\begin{equation}\label{fhuK}
f\circ h_{|V}=u\prod_{i\in I} y_{i}^{N_{i}},
\end{equation}
with $y_i\in\cO_Y(V)$ an equation of $E_i$ in $V$ and $u\in\cO_Y(V)$ invertible, then the images of
$(y_{i})_{i\in I}$ in the stalk
$\mathcal{O}_{Y,a}$ at $a$ are part of a regular system of parameters.
\item The relative canonical divisor $K_{Y/X_K}$, which is locally defined by the determinant of the Jacobian matrix of $h$, is written as
$\sum_{j\in T_K} (\nu_j-1)E_j$ for some positive integers $\nu_j$.
\end{itemize}
For any subset $I\subseteq T_K$, we put $E_I := \bigcap_{i\in I} E_i$
if $I$ is non-empty and $E_{I}  = Y$ if $I$ is the empty set.
Further, we write
$$
E_{I}^{\circ} := E_I \smallsetminus \bigcup_{i\not \in I} E_i.
$$

By the functoriality of log resolutions for extensions of the base field,
for any field $K'$ containing $K$,
$h$ induces a log resolution $h_{K'}\colon Y_{K'} \to X_{K'}$ of
the pair $(X_{K'},D_{K'})$.
We note that each irreducible component $E_{i}$ for $i\in T_K$ splits into a disjoint union of finitely many irreducible components $E_{ij}$ over $K'$ with $(i,j)\in T_{K'}$ for a corresponding finite set $T_{K'}$, and where we always have $N_{i}=N_{ij}$ and $\nu_{i}=\nu_{ij}$.

When $K'=\CC$, we write  $J$ for
$T_{K'}$. We say that $I'\subseteq J$ corresponds to $I\subseteq T_K$ if $I'$ \change{ranges over} the irreducible  components (over $\CC$) of the $E_i$ for $i\in I$.

\begin{defn}[Power condition for $(f,h,Z)$] \label{defn:pow}
Suppose now that $K=\CC$ and let
 $f$ and  $h\colon Y\to X$ be as above \change{with $X=X_\CC$}. Let $Z$ be a closed subvariety of $X$ such that $f$ vanishes on $Z(\CC)$.
Consider a non-empty open subset $W$ of an irreducible component of $E_I$ for some $I\subseteq J$,
let $g$ be in $\mathcal{O}_W(W)$, and $d>1$ be an integer.
We say that $(f,h,Z)$ satisfies the power condition, witnessed by  $(I,W,g,d)$, if the following conditions hold:
\begin{equation}\label{eq:ugd0}
h(W)\subseteq Z,
\end{equation}
\begin{equation}\label{eq:ugd00}
d|N_{i} \mbox{ for all } i\in I,
\end{equation}
and
\begin{equation}\label{eq:ugd}
u|_{W}=g^{d},
\end{equation}
where $u$ comes from writing $f\circ h=u\prod_{i\in I} y_{i}^{N_{i}}$ as in (\ref{fhuK}) on an open subset $V\subseteq Y$ with $W=E_I\cap V$.
We simply say that the power condition holds for $(f,h,Z)$ if there exists  $(I,W,g,d)$ witnessing the power condition for $(f,h,Z)$.
\end{defn}

\subsection{}

Before giving the proof of Theorem \ref{thm_bound}, we make a few preliminary remarks, using the notation in \S\ref{subsec:bound}.

\begin{remark}\label{rmk2_thm}
In order to prove the inequality  (\ref{eq_thm}) in Theorem \ref{thm_bound}, it is enough to consider the case when $I$ has only one element. Indeed, given an arbitrary subset $I$
as in the theorem, let $Z$ be a connected component of $E_I$ that meets $X_0$, and let $h'\colon Y'\to Y$ be the blow-up of $Y$ along $Z$,
with exceptional divisor $F$. Note that in this case we have
$$N={\rm ord}_F(f)=\sum_{i=1}^mN_i\quad\text{and}\quad \nu={\rm ord}_F(K_{Y'/X})+1=\sum_{i=1}^m\nu_i.$$
Consider the chart $X'_0\subseteq {h'}^{-1}(X_0)$ on $Y'$ with coordinates $y_1,\ldots,y_n$ such that $x_1=y_1$, $x_i=y_1y_i$ for $2\leq i\leq m$
and $x_i=y_i$ for $i>m$. In this case, we have
$$f\circ h\circ h'\vert_{X'_0}=y_1^N\cdot \prod_{i=2}^my_i^{N_i}\cdot (u\circ h').$$
If $g'=g\circ h'\in\cO(X'_0\cap F)$, then
$u\circ h'\vert_{X'_0\cap F}={g'}^d$ and clearly $d$ divides $N$. If we know (\ref{eq_thm}) in the case of one divisor, we obtain
$$\lct (f)\leq\frac{1}{d}+\sum_{i=1}^m\nu_i-\lct (f)\cdot\sum_{i=1}^mN_i=\frac{1}{d}+\sum_{i=1}^mN_i\left(\frac{\nu_i}{N_i}-\lct (f)\right),$$
hence (\ref{eq_thm}) holds.

From now on, we will thus assume that $I$ contains only one element, corresponding to the divisor $E$ on $Y$,
and denote by $N$ and $\nu$ the corresponding invariants. Note that in this case,
the inequality in (\ref{eq_thm}) is equivalent to
\begin{equation}\label{eq2_thm}
\lct (f)\leq\frac{d\nu+1}{d(N+1)}.
\end{equation}
\end{remark}

\begin{remark}\label{rmk_indep_res}
It is interesting to note that in the case of one divisor $E$, the hypothesis on $f$ is independent of the log resolution $h$ and only depends on the
valuation $v={\rm ord}_E$ corresponding to $E$. Indeed, if $\cO_v$ is the corresponding DVR, with residue field $k_v$, and if we write
$f=\pi^Nu$, where $\pi$ is a uniformizer of $\cO_v$, $N=v(f)$, and $u\in\cO_v$ is invertible, then the condition on $f$ is that the class
$\overline{u}\in k_v^{\times}$ lies in $(k^{\times}_v)^d$ for some $d>1$ with $d\vert N$, where $(k^{\times}_v)^d$ is the set of $d$-th powers in $k_v^\times$.
\end{remark}

\begin{remark}
It is enough to prove (\ref{eq_thm}), since it implies (\ref{eq_thm_Z}) for any $Z$ that satisfies the conditions in the theorem.
Indeed, arguing as in Remark~\ref{rmk2_thm}, we see that it is enough to consider the case when $I$ consists of only one element, corresponding to the divisor $E$,
in which case the condition is that $h(E)\cap Z\neq\emptyset$. If $V$ is an open neighborhood of $Z$ such that
$\lct _Z(f)=\lct (f\vert_V)$, then $E\cap h^{-1}(V)\neq\emptyset$, and we may apply the theorem for the restriction of $h$ over $V$
to obtain our assertion.
\end{remark}

\begin{remark}\label{rmk1_thm}
We may and will assume that $E$ is an exceptional divisor.
Indeed, otherwise we have $\nu=1$ and
$$\lct (f)\leq\frac{1}{N}\leq\frac{d+1}{d(N+1)},$$
where the second inequality follows from the fact that by assumption we have $d\leq N$.
\end{remark}

\begin{remark}\label{rmk2.5_thm}
The inequality (\ref{eq2_thm}) clearly holds if the right-hand side is $\geq 1$. We thus may and will assume that
$\frac{d\nu+1}{dN+d}<1$, hence $\nu<N+\frac{d-1}{d}$. Since both $\nu$ and $N$ are integers, this implies $\nu\leq N$.
\end{remark}

\begin{remark}\label{rmk3_thm}
Furthermore,
we may also assume that there is a rational number $c$, with $0<c<\lct (f)$, such that
\begin{equation}\label{eq_rmk3_thm}
\nu-c\cdot N<1.
\end{equation}
The existence of such $c$ is clear if $\frac{\nu-1}{N}<\lct (f)$. On the other hand, if
$\lct (f)\leq\frac{\nu-1}{N}$, then we are done since it is easy to check that we have
$$\frac{\nu-1}{N}<\frac{d\nu+1}{dN+d}$$
when $\nu\leq N$.

The existence of such $c$ as above is useful since it implies that there is a projective, birational morphism $\pi\colon W\to X$, with $W$ normal,
such that $E$ appears as a prime $\QQ$-Cartier divisor $E_W$ on $W$, and such that $E_W$ is the unique exceptional divisor of $\pi$. This is a well-known consequence
of the Minimal Model Program: note that the pair $(X,cD)$ is klt and we can apply
\cite[Corollary~1.4.3]{BCHM} or \cite[Propositions~3.2 and 4.1]{Blum}.
Furthermore, since $X$ is smooth, hence $\QQ$-factorial,
the exceptional locus of $\pi$ has pure codimension $1$, hence it is equal to $E_W$. Note that while $W$ is not a log resolution of $(X,D)$, the hypothesis in
the theorem is birational with respect to the divisor, hence it also holds for $E_W$ (see Remark~\ref{rmk_indep_res}).
\end{remark}

The assertion in the following lemma is well-known, but we include a proof for the sake of completeness.

\begin{lem}\label{lem_terminal}
If $X$ is a normal, ${\mathbb Q}$-Gorenstein variety, and $Z\subseteq X$ is a codimension 2 irreducible closed subset, such that $X$ is not smooth
at the generic point of $Z$, then there is a projective, birational morphism $\pi\colon \widetilde{X}\to X$, with $\widetilde{X}$ smooth, and a prime divisor $F$ on $\widetilde{X}$ such that
$\pi(F)=Z$ and the coefficient of $F$ in $K_{\widetilde{X}/X}$ is $\leq 0$.
\end{lem}

\begin{proof}
Let $\pi\colon \widetilde{X}\to X$ be a log resolution, with exceptional divisor $F_1+\ldots+F_r$.
It follows from \cite[Corollary~2.32]{KollarMori} that if $Y$ is the union of those $\pi(F_i)$ such that $F_i$ has \change{coefficient $\leq 0$} in
$K_{\widetilde{X}/X}$ and $U=X\smallsetminus Y$, then
$U$ has terminal singularities. In particular, $U$ is smooth in codimension 2 (see
\cite[Corollary~5.18]{KollarMori}), hence our assumption implies that $Z\subseteq Y$. Therefore $Z$ is an irreducible component of $Y$,
hence it is equal to $\pi(F_i)$ for some divisor $F_i$ whose coefficient in $K_{\widetilde{X}/X}$ is $\leq 0$.
\end{proof}

Finally, we will need the following bound for the intersection multiplicity of two curves.

\begin{lem}\label{bound_int_multiplicity}
Let $(R,\frm_R)$ be a local, excellent domain, with $\dim(R)=2$. Suppose that $g\in \frm_R$ is a non-zero element that generates a prime ideal
and $h\in\frm_R$ is such that its image $\overline{h}$ in $A=R/(g)$ is non-zero and can be written as $u^d$, for some $u$ in the fraction field of $A$.
In this case we have
$$\ell_R\big(R/(g,h)\big)\geq d.$$
\end{lem}

\begin{proof}
Let $B$ be the integral closure of $A$ in its fraction field. Since $R$ is excellent, $B$ is a finitely generated $A$-module. For a finitely generated $A$-module $M$
and an ideal $\frq$ in $A$, with radical equal to the maximal ideal, we write $e_A(\frq,M)$ for the Samuel multiplicity of $M$ with respect to $\frq$.
Since ${\rm rank}_A(B)=1$, it follows from \cite[Theorem~14.8]{Matsumura} that
$$e_A\big((\overline{h}),B\big)=e_A\big((\overline{h}),A\big).$$
Note also that we have
$$e_A\big((\overline{h}),A\big)=\ell_A\big(A/(\overline{h})\big)=\ell_R\big(R/(g,h)\big),$$
where the first equality follows from the fact that $\overline{h}$ is a non-zero-divisor in the 1-dimensional local ring $A$ (see \cite[Theorem~14.11]{Matsumura}).
Using the fact that $\overline{h}$ and $u$ are non-zero-divisors in $B$, we also have
$$e_A\big((\overline{h}),B\big)=\ell_A(B/\overline{h}B)=d\cdot \ell_A(B/uB)\geq d.$$
This completes the proof of the lemma.
\end{proof}

We can now give the proof of the bound for the log canonical threshold.

\begin{proof}[Proof of Theorem~\ref{thm_bound}]
We may and will assume that we are in the situation described in Remark~\ref{rmk3_thm}, with a morphism $\pi\colon W\to X$ whose
exceptional locus is equal to $E_W$, the prime divisor on $W$ corresponding to $E$. By hypothesis, $h$ is an isomorphism
over $X\smallsetminus D$, hence
$\pi(E_W)\subseteq D$. For every $y\in\pi(E_W)$, the fiber $\pi^{-1}(y)$ is contained in the exceptional locus. Since $\pi$ is proper, it follows that
$\pi^{-1}(y)\cap\widetilde{D}\neq\emptyset$, where $\widetilde{D}$ is the strict transform of $D$
on $W$. In particular,
$E_W\cap\widetilde{D}$ is non-empty.
Note that we have $\pi^*(D)=\widetilde{D}+NE_W$ and $E_W$ is ${\mathbb Q}$-Cartier, hence the divisor $\widetilde{D}$ is $\QQ$-Cartier. Similarly, since $K_{W/X}=(\nu-1)E_W$,
it follows that $K_{W/X}$ is $\QQ$-Cartier.

Let $c=\lct (f)$, so that $(X,c D)$ is log canonical.
 This implies that also the pair
 $$(W,c\cdot \pi^*(D)-K_{W/X})=(W,c_1\widetilde{D}+c_2E_W)$$ is log canonical,
 where $c_1=c$ and $c_2=cN-\nu+1$.
Arguing by contradiction, we may assume that $c>\frac{d\nu+1}{dN+d}$.
In this case we have
$$c_1+c_2=c+(cN-\nu+1)=c(N+1)-\nu+1>\frac{d\nu+1}{d}-\nu+1=1+\frac{1}{d}.$$
In particular, this gives $c_2>\frac{1}{d}>0$.

We now consider a suitable cyclic cover.
Let $s$ be a positive integer such that $sE_W$ is Cartier
and choose an open subset $V$ of $W$ meeting $\widetilde{D}\cap E_W$ such that
we have an isomorphism $\cO_V(sE_V)\simeq\cO_V$, where $E_V=E_W\vert_V$.
After possibly replacing $s$ by a divisor and $V$ by a smaller open subset, we may assume that
$s'E_V$ is not Cartier for any divisor $s'$ of $s$ different from $s$.
 Consider the $\cO_V$-algebra
$${\mathcal A}=\cO_V\oplus\cO_V(E_V)\oplus\ldots\oplus\cO_V\big((s-1)E_V\big),$$
where multiplication is defined using the fact that for $0\leq i,j\leq s-1$ with $i+j\geq s$, we have
$$\cO_V(i E_V)\otimes\cO_V(j E_V)\to \cO_V\big((i+j)E_V\big)\simeq \cO\big((i+j-s)E_V\big).$$
Note that we have a finite surjective morphism $\phi\colon U={\mathcal Spec}({\mathcal A})\to V$.
It is well-known and straightforward to check that $U$ is normal and $\phi$ is \'{e}tale in codimension $1$; in particular,
we have $K_{U/V}=0$.
Moreover, the section $1$ of $\cO_V(E_V)$ defines an effective Cartier divisor $E_U$ on $U$ such that
$\phi^*(E_V)=E_U$.
We also put $D_U=\phi^*(\widetilde{D})$. Note that since $\widetilde{D}+NE_W$ is Cartier and $E_U$ is Cartier,
it follows that also $D_U$ is Cartier, as well. Furthermore, since $\widetilde{D}\cap E_V\neq\emptyset$, we conclude that $D_U\cap E_U\neq\emptyset$.
Let $Z$ be an irreducible component of $D_U\cap E_U$, so that $Z$ has codimension $2$ in $U$.

We first show that $U$ is smooth at the generic point of $Z$. Indeed,
since $Z$ has codimension $2$ in $X$,
if $U$ is not smooth at the generic point of $Z$,
it follows from Lemma~\ref{lem_terminal} that
there is a prime divisor $F$ on some smooth variety $\widetilde{U}$, with a birational morphism $\widetilde{U}\to U$, such that $F$ dominates $Z$ and
${\rm ord}_F(K_{\widetilde{U}/U})\leq 0$. Since $K_{U/V}=0$ and the pair
$(W,c_1\widetilde{D}+c_2E_W)$ is log canonical, it follows
that the pair $(U,c_1D_U+c_2E_U)$ is log canonical, and thus
$$1\geq 1+{\rm ord}_F(K_{\widetilde{U}/U})\geq c_1\cdot {\rm ord}_F(D_U)+c_2\cdot {\rm ord}_F(E_U)\geq c_1+c_2>1+\frac{1}{d},$$
a contradiction. Therefore $U$ is smooth at the generic point of $Z$.

Note that since $\phi$ is \'{e}tale in codimension $1$, the divisor $E_U$ is reduced. Let $E_U^1,\ldots,E_U^r$ be the prime
divisors containing $Z$ that appear in $E_U$. If $R=\cO_{U,Z}$, then the image of $f$ in $R$ factors as
$h\cdot\prod_{i=1}^rg_i^N$, where
$h\in R$ a local equation of $D_U$, and $g_1,\ldots,g_r\in R$
are local equations of $E_U^1,\ldots,E_U^r$.
By considering the local homomorphism $\cO_{V,E_V}\hookrightarrow\cO_{U,E_U^1}$, we deduce from the hypothesis in the theorem
that the class of $h\cdot\prod_{i=2}^r g_i^N$ in $k(E_U^1)$ is the $d^{\rm th}$ power of some element of $k(E_U^1)$; therefore the same holds for
the class of $h$ in $k(E_U^1)$. We may thus apply
 Lemma~\ref{bound_int_multiplicity}  to conclude that
\begin{equation}\label{eq_bound_length}
\ell_R\big(R/(g_1,h)\big)\geq d.
\end{equation}

On the other hand,
since $U$ is smooth at the generic point of $Z$, we have the divisorial valuation ${\rm ord}_Z$ of the function field of $U$
\change{(this corresponds to the exceptional divisor on the blow-up along $Z$ of a smooth open subset of $U$ meeting Z)}.
Since the pair $(U, c_1D_U+c_2E_U)$ is log canonical, we have
$$2\geq c_1\cdot {\rm ord}_Z(D_U)+c_2\cdot {\rm ord}_Z(E_U).$$
Since $c_1+c_2>1$, we conclude that either ${\rm ord}_Z(E_U)=1$ or ${\rm ord}_Z(D_U)=1$.
We treat these two cases separately.

\noindent {\bf Case 1}.  Suppose that ${\rm ord}_Z(E_U)=1$, that is, $E_U$ is smooth at the generic point of $Z$. In particular, we have $r=1$.
In this case there is a regular system of parameters of $R$ given by $g_1$ and some $x\in R$. Let $v$ be the monomial valuation (with respect to this
coordinate system) of the fraction field of $R$, such that $v(g_1)=d$ and $v(x)=1$. Condition (\ref{eq_bound_length}) implies that $h\in (g_1, x^d)$, hence $v(h)\geq d$.
It is a standard fact that the log discrepancy of $v$ is $d+1$, hence the fact that $(U,c_1D_U+c_2E_U)$ is log canonical implies
$$d+1\geq c\cdot v(h)+(cN-\nu+1)\cdot v(g_1)\geq d(c+cN-\nu+1).$$
A straightforward computation then gives
$$c\leq\frac{d\nu+1}{dN+d},$$
completing the proof of this case.

\noindent {\bf Case 2}. If ${\rm ord}_Z(D_U)=1$, we proceed similarly. Consider a regular system of parameters of $R$ given by $h$ and $y$
and consider the monomial valuation $w$ (in this system of coordinates) of the fraction field of $R$ such that $w(h)=d$ and $w(y)=1$.
It follows from (\ref{eq_bound_length}) that $g_1\in (h,y^d)$, hence $w(g_1)\geq d$. Since the log discrepancy of $w$ is $d+1$, using the fact that
the pair $(U,c_1D_U+c_2 E_U)$ is log canonical, we obtain
$$d+1\geq c\cdot w(h)+(cN-\nu+1)\cdot w(g_1)\geq dc+d(cN-\nu+1),$$
which again implies
$c\leq\frac{d\nu+1}{dN+d}$,
completing the proof of the theorem.
\end{proof}

\subsection{}

For our applications we will use the following corollary of Theorem \ref{thm_bound}, with notation as above and with $X=\AA^n_\CC$.
\begin{cor}
\label{cor:num}
Let $f$ be a non-constant polynomial in $\CC[x]$, $h\colon Y\to \AA_{\CC}^n$ a log resolution of the pair $(\AA^n_\CC,D)$ with $D$ given by $f$
and let $Z$ be a closed subvariety of $\AA_{\CC}^n$ such that $f$ vanishes on $Z(\CC)$. 
If the power condition holds for $(f,h,Z)$, witnessed by some $(I,W,g,d)$,
then the following inequality holds
\begin{equation}\label{eq:Sub:weak:conj}
\lct_Z(f)\leq\dfrac{1}{d} + \sum_{i\in I}\big(\nu_{i}-N_i\cdot \lct_Z(f)\big).
\end{equation}
\end{cor}

Note that the weaker inequality
\begin{equation}\label{eq_thm_Z_1/2}
\lct _Z(f)  \leq \frac{1}{2} + \sum_{i\in I}\left(\nu_i- N_i\cdot \lct _Z(f)\right),
\end{equation}
with 1/2 instead of the term $1/d$ in  (\ref{eq:Sub:weak:conj}), would already suffice to prove our results on exponential sums, but not for our application on poles with largest possible multiplicity in section \ref{subsec:poles}.

\section{Igusa's local zeta function and exponential sums}\label{igusa}

\subsection{}\label{sec:Lchi}
Let $f$, $Z$ and $\cO$ be as in Theorem \ref{thm:sum}.
Consider a local field $L$ over $\cO$ and let $\chi\colon \cO_{L}^{\times}\rightarrow\CC^{\times}$ be a multiplicative character, that is, a continuous group homomorphism on the group of units, $\cO_{L}^{\times}$, of $\cO_{L}$. Note that any such $\chi$ has finite image. The \emph{order} of $\chi$ is the number of elements in its image. The \emph{conductor} $c(\chi)$ of $\chi$ is the smallest $c\geq 1$ for which $\chi$ is trivial on $1+\cM_L^c$, with $\cM_L$ the maximal ideal of $\cO_L$. We put $\chi(0)=0$. Let $s$ be a complex number with real part at least $0$. With a fixed uniformizer $\varpi_L$ of $\cO_L$, we consider the map $\acc\colon L\to \cO_L$ that sends a nonzero $x$ to $x\varpi_L^{-\ord x}$ and $0$ to $0$. Further, write $\ac(x)$ in $k_L$ for the reduction of $\acc(x)$ modulo $\cM_L$.
We now associate to this data Igusa's \emph{local  zeta function}
\begin{equation}\label{eq:ZchiPhi}
\cZ^Z_{f,L,\chi,s}:=\int_{{\{x\in \cO_{L}^n\mid \overline x\in Z(k_L) \} }}\chi\big(\acc(f(x))\big)|f(x)|^{s}|dx|.
\end{equation}
Igusa showed in \cite{Igusa1} that $\cZ^Z_{f,L,\chi,s}$ is a rational function in $t=q_L^{-s}$ when $L$ has characteristic zero (and, when $L$ has positive, large enough characteristic, given $f$), thus starting the study of a now vast subject.

We begin by recalling a result relating exponential sums to Igusa's local zeta functions (see \cite[Proposition 1.4.4]{DenefBour}).

\begin{prop}\label{004}
Let $f$, $\cO$, and $Z$ be as in Theorem \ref{thm:sum}, $L$ a local field over $\cO$, and $\psi$ a nontrivial additive character on $L$.
If we put $m=m_\psi$, $q=q_L$, and $t=q^{-s}$, then
$E^Z_{f,L,\psi}$ is equal to
\begin{equation}\label{eq:1.4.4-Bour}
 \cZ^Z_{f,L,1,0} + \mbox{\textnormal{Coeff}}_{t^{m-1}}\Big(\dfrac{(t-q)\cZ^Z_{f,L,1,s} }{(q-1)(1-t)}\Big)
+\sum_{\chi\neq 1 }g_{\chi^{-1},\psi}\mbox{\textnormal{Coeff}}_{t^{m-c(\chi)}}\big(\cZ^Z_{f,L,\chi,s} \big),
\end{equation}
where $1$ stands for the trivial character on $\cO_L^\times$,  the summation index $\chi$ runs over all nontrivial multiplicative characters on $\cO_L^\times$,
$g_{\chi,\psi}$ is a complex number depending only on $\chi$ and $\psi$, and where $\mbox{\textnormal{Coeff}}_{t^{\ell}}S(t)$ for any $\ell\geq 0$ and any power series $S$ in $t$ stands for the coefficient of $t^\ell$. Moreover,  if $c(\chi)=1$, then
\begin{equation}\label{eq:|gchi|}
|g_{\chi,\psi}| = \frac{q^{1/2}}{q-1}.
\end{equation}
\end{prop}

For an explicit description of the $g_{\chi,\psi}$ in (\ref{eq:1.4.4-Bour}),
see \cite[Proposition~1.4.4]{DenefBour}, whose proof applies to local fields of any characteristic.

We recall a variant of the Lang-Weil estimates and a corollary. 
\begin{prop}[Lang-Weil estimates]\label{Lang-Weil}
Let $k=\FF_q$ be a finite field and $X\subseteq \PP_k^n$ be a closed subvariety of dimension $r$. If $X$ is geometrically irreducible, then
there is a positive constant $c_X$ such that for every $\ell\geq 1$ we have
\[
|\#X(\FF_{q^\ell})-q^{\ell r}\mid \leq
c_X q^{\ell (r-\frac{1}{2})}.
\]
Moreover, $c_X$ can be taken independently from $X$ and from $q$ as long as $n$, $r$, and the number and degrees of the equations defining $X$ remain bounded.
\end{prop}
\begin{proof}
The existence of $c_X$ comes from the usual Lang-Weil estimates. The independence of $c_X$ from $X$ and from $q$ (as long as the complexity of $X$ stays bounded), follows from \cite[Theorem~12]{Katz-Bet}, which gives furthermore explicit upper bounds for $c_X$ in terms of the complexity of $X$ (see also \cite[Theorem~3.1]{Rojas-Leon}).
\end{proof}

\begin{cor}\label{sums1/2}\label{nontrivial}
Let  $\cO$ be a ring of integers and let $d>1$ be an integer. Let $X\subseteq \AA^n_\cO$ be a closed subscheme such that $X_\CC$ is an irreducible closed subvariety of $\AA^n_\CC$ of dimension $r$, and let $F\colon X\to \AA^1_\cO$ be a regular morphism \change{such that $F$ is nonvanishing on $X(\CC)$}. Suppose that there does not exist $e>1$ dividing $d$ and a regular morphism $g\colon V\subset X_\CC \to\AA_\CC^1$ on a non-empty open $V$ of $X_\CC$ such that $g^e$ equals $F\vert_{V}$.
Then there exist constants $c$ and $M$ such that for all finite fields $\FF_q$ of characteristic at least $M$ with $\FF_q$ an algebra over $\cO$, and for any character $\chi$ of $\FF_q^\times$ of order $d$, we have
\begin{equation}\label{eq:chi1/2}
 | \sum_{x\in X(\FF_q)} \chi \big(F(x)\big) | \leq c q^{r-1/2}.
 \end{equation}
Moreover, $c$ can be taken independently from $X$ and $F$ as long as $n$, $r$, \change{$d$,} and the number and degrees of the equations defining $X$ and $F$ remain bounded.
\end{cor}
\begin{proof}
Let $U$ be the Kummer cover of $X$  given by $F(x)=y^d$ for $x\in X$. By our assumptions, we have that $U_\CC$ is irreducible, and thus there exist $M$ and $c$ such that for each finite field $\FF_q$, which is an algebra over $\cO$ and whose characteristic is at least $M$,  we have that $U_{\FF_q}$ is geometrically irreducible (e.g.~by model theoretic compactness)  and, by Proposition \ref{Lang-Weil}, that
$$
|\# U(\FF_q) - q^{\dim X}| \leq c q^{\dim X-1/2}.
$$
Write $\FF_q^{\times,d}$ for the set of $d$th powers in $\FF_q^\times$ and  $d_q$ for the index of $\FF_q^{\times,d}$ in $\FF_q^\times$.
Clearly, we also have
$$
\frac{\# U(\FF_q)}{d_q} =\#\{x\in X(\FF_q)\mid F(x) \mbox{ is a $d$th power in } \FF_q^\times \}.
$$
Similarly, for each such $q$ and for each $\lambda\in \FF_q^\times$, we consider $U_\lambda$  given by $F(x)=\lambda y^d$. By the uniformity of the constant in Proposition \ref{Lang-Weil}, we can choose $M$ and $c$ as above and such that, in addition,
 if the characteristic of $\FF_q$ is at least $M$ (and if $\FF_q$ is an algebra over
$\cO$), then for each $\lambda\in \FF_q^{\times}$ we have that $U_{\lambda}$ is geometrically irreducible (again, by model theoretic compactness) and
$$
|\#\{x\in X(\FF_q)\mid  \frac{F(x)}{\lambda} \mbox{ is a $d$th power in } \FF_q^\times  \} -  \frac{ q^{\dim X}}{d_q}|  \leq \change{\frac{c}{d_q}} q^{\dim X-1/2}.
$$
By orthogonality of characters, for any character $\chi$ of $\FF_q^\times$ of order $d$, we have
$$
 \sum_{\overline{\lambda} \in \FF_q^\times/\FF_q^{\times,d} }  \chi(\overline{\lambda})=0,
 $$
\change{from which the corollary follows. Indeed, the required uniformity of $c$ comes from the uniformity in Proposition \ref{Lang-Weil} and the fact that
 the complexity of the covers $U_\lambda$ is clearly bounded when $n$, $r$, $d$, and the number and degrees of the equations defining $X$ and $F$ are bounded.} 
\end{proof}

We next give a combination of Denef's formula for Igusa's local zeta function, the Lang-Weil estimates, and Corollary \ref{nontrivial}.
Possibly one may use the more advanced estimates of \cite[Theorem~1.1]{Rojas-Leon} on finite field exponential sums with multiplicative characters instead of Corollary \ref{nontrivial}.
Let $K$ be the field of fractions of $\cO$. With the notation in \S\ref{sec:resolK}, with $X={\mathbb A}^n_K$ and $f\in K[x]=K[x_1,\ldots,x_n]$
non-constant,
we consider $D_{K}$ and fix a log resolution $h\colon Y\to \AA^n_K$ of the pair  $(\AA_K^n,D_{K})$.

\begin{prop}
\label{005}
Let $f$, $Z$, $\cO$, and $h$ as above. Assume moreover that $f$ vanishes on $Z(\CC)$.
Then there exist constants $C$ and $M$ so that the following formula
holds for every local field $L$ over $\cO$ with residue field characteristic at least $M$
and every multiplicative character $\chi$ on $\cO_{L}^{\times}$:
\begin{equation}\label{eq:ZPhichi}
\cZ^Z_{f,L,\chi,s}=
\sum_{\substack{I\subseteq T_{K}}}
c_{I,Z,L,\chi} \cdot \prod_{i\in I}\frac{q_L^{-N_{i}s-\nu_{i}}}{1-q_L^{-N_{i}s-\nu_{i}}},
\end{equation}
where the complex numbers $c_{I,Z,L,\chi}$ are independent of $s$ and satisfy
\begin{equation}\label{eq:cI1}
|c_{I,Z,L,\chi}|\leq C.
\end{equation}
Moreover, for such $L$ we further have
\begin{equation}\label{eq:cI2}
c_{I,Z,L,\chi} = 0
\end{equation}
if $c(\chi)>1$ or if the order of $\chi$ does not divide $N_i$ for some $i\in I$.
Furthermore,  if $\chi$ is nontrivial and for $I'\subseteq J$ corresponding to $I$
there exist no $W,g,d$ with $(f,h,\AA^n)$ satisfying the power condition witnessed by $(I',W,g,d)$, then
\begin{equation}\label{eq:cI3}
|c_{I,Z,L,\chi}|\leq C q_L^{-1/2}.
\end{equation}
Finally, given $f$ and $h$, the constants $C$ and $M$ can be taken independently from $Z$ and $\cO$, as long as the number of equations defining $Z$ and their degrees remain bounded.
\end{prop}
\begin{proof}
By \cite[Theorem~2.1]{Den5} (or, equivalently, \cite[Theorem~3.3]{DenefBour}), if $c(\chi)>1$, and since $f$ vanishes on $Z(\CC)$, we have $c_{I,Z,L,\chi}=0$ for all $I$ and all $Z$ (as soon as the residue field characteristic is large enough) and we are done for such $\chi$.
Let us thus take $\chi$ with $c(\chi)=1$. The existence of the complex numbers $c_{I,Z,L,\chi}$
such that (\ref{eq:ZPhichi}) holds, as well as their independence of $s$ follow from Denef's formula \cite[Theorem~2.2]{Den5}
 (or, equivalently, \cite[Theorem~3.4]{DenefBour}), where also explicit descriptions of the  $c_{I,Z,L,\chi}$ are given as finite exponential sums (as soon as the residue field characteristic is large enough). Precisely, the explicit description of the $c_{I,Z,L,\chi}$ given in \cite[Theorem~2.1]{Den5} (or \cite[Theorem~3.3]{DenefBour})
is as follows:
\begin{equation}\label{eq:c-denef}
c_{I,Z,L,\chi}  = \frac{(q_L-1)^{ \# I}}{q_L^n} \sum_{a\in E_I^\circ (k_L),\ h(a)\in Z(k_L)}\chi(u(a))
\end{equation}
if the order of $\chi$ divides $N_i$ for each $i\in I$ and if the characteristic of $k_L$ is sufficiently large (depending only on $f$ and $h$),
where we take natural reductions modulo the maximal ideal $\cM_L$ of $\cO_L$ when we write $u(a)$, $h(a)$ and $E_I^\circ (k_L)$.
 The bound (\ref{eq:cI1}) now follows from the Lang-Weil estimates for bounding the number of elements in $E_I^\circ (k_L)$.
Further, if $I$ is such that the order of $\chi$ does not divide $N_i$ for some $i\in I$, then $c_{I,Z,L,\chi}=0$ by \cite[Theorem~2.2]{Den5}
(or \cite[Theorem~3.4]{DenefBour}), still assuming that the residue field characteristic is large enough. This proves (\ref{eq:cI2}).

We still need to show (\ref{eq:cI3}). Suppose thus that $\chi$ is nontrivial and, given $I$,  that there does not exist $W,g,d$ such that $(f,h,\AA^n)$ satisfies the power condition witnessed by $(I',W,g,d)$,  with $I'$ corresponding to $I$.
For those $L$ such that the reduction of $E_I$ modulo $\cM_L$ has no irreducible component (defined over $k_L$) which is moreover geometrically irreducible, we reason as follows. By the smoothness of $E_I$ we have that the reduction of $E_I$ modulo $\cM_L$ is also smooth, as soon as $p_L$ is large.  Hence, if no irreducible component of the reduction of $E_I$ modulo $\cM_L$  is moreover geometrically irreducible, then there are no $k_L$-rational points on the reduction of $E_I$ by its smoothness, and, (\ref{eq:cI3}) is clear.
Suppose now that there is an irreducible component of $E_{I,k_L}$  (the reduction of $E_I$ modulo $\cM_L$) which is geometrically irreducible.
By working separately for each component of $E_{I,k_L}$, and using the Lang-Weil estimates
in order to see that we can ignore  algebraic subsets of codimension at least $1$, we can choose an affine open $V$ of $E_I^\circ$ and restrict the summation index in (\ref{eq:c-denef}) by imposing both $a\in V(k_L)$ and $h(a)\in Z(k_L)$.
Furthermore, since the power condition for $(f,h,\AA^n)$ does not hold for any witnesses  of the form $(I',W,g,d)$, with $I'$ corresponding to $I$ and any integer $d>1$,
we may apply Corollary \ref{nontrivial} to the sum in (\ref{eq:c-denef}) restricted to $a\in V(k_L)$ to find the bound from (\ref{eq:cI3}) in the case that
$V_{k_L}\cap h^{-1}(Z_{k_L})$ has dimension equal to $\dim E_I$, and Proposition \ref{Lang-Weil} in the case that its dimension  is less than $\dim E_I$, with $V_{k_L}$ and $Z_{k_L}$ denoting the reductions.  The proposition is proved: note that the uniformity of $C$ and $M$ for varying $Z$ and $\cO$  follows from the uniformity of Proposition \ref{Lang-Weil} and Corollary \ref{nontrivial}.
\end{proof}

\subsection{}

Our strategy for proving Theorem \ref{thm:sum} is to first reduce to the case when $f$ vanishes on $Z(\CC)$. Next, we relate the exponential sums to Igusa's local zeta functions using Proposition \ref{004}, and finally we estimate the different parts in (\ref{eq:1.4.4-Bour}) of Proposition \ref{004} using Proposition \ref{005}, Corollary \ref{cor:num}, and the following two propositions.
 Let $f\in\cO[x]$ and $Z$ be as in Theorem \ref{thm:sum} and let $K$ be the number field  which is the field of fractions of $\cO$. We fix a log resolution $h:Y\to\AA^n_K$ of $D_{K}$.
 In order to estimate the first term in (\ref{eq:1.4.4-Bour}), we use the following result from
  \cite{Saskia-Kien}.

\begin{prop}[\cite{Saskia-Kien}, Lemma 4.1]\label{trivial}Given $f$ and $h$ as above, there exist positive constants $C$ and $M$ such that, for any integer $m\geq 2$, any ring of integers $\cO_1$ containing $\cO$, any closed subscheme $Z$  of $\AA^n_{\cO_1}$ such that $f$ vanishes on $Z(\CC)$, and any local field $L$ over $\cO_1$ with $p_L$ larger than $M$, we have
$$
|\cZ^Z_{f,L,1,0}+\textnormal{Coeff}_{t^{m-1}}\dfrac{(t-q)\cZ^Z_{f,L,1,s}}{(q-1)(1-t)}|
\leq Cm^{n-1}q^{-m\cdot \lct_Z(f)},
$$
where $q=q_L$ and $t=q^{-s}$.
\end{prop}
\begin{proof}
Lemma 4.1 of \cite{Saskia-Kien} states and proves this for $Z=\{0\}$, but in its proof one can replace this $\{0\}$ by any choice of $Z$ such that $f$ vanishes on $Z(\CC)$, and $\lct_0(f)$ by $\lct_Z(f)$. Indeed, the estimates (4.7) and  (4.9) of \cite{Saskia-Kien}  are valid for any $Z$ such that $f$ vanishes on $Z(\CC)$, instead of $\{0\}$, and with $\lct_Z(f)$ instead of $\lct_0(f)$.
\end{proof}
We now estimate the remaining parts in (\ref{eq:1.4.4-Bour}) of Proposition \ref{004}.

\begin{prop}\label{estimation2}\label{estimation1}
Let $f$, $\cO$, and $Z$ be as in Theorem \ref{thm:sum}.
If $I\subseteq T_K$ is such that $h\big(E_I(\CC)\big)\cap Z(\CC)$ is non-empty,
then for every
$q>1$ and every $m\geq 2$, we have
\begin{equation}\label{eq:est2}
|\mbox{\textnormal{Coeff}}_{t^{m-1}}(\prod_{i\in I}\frac{t^{N_{i}}q^{-\nu_{i}}}{1-t^{N_{i}}q^{-\nu_{i}}})|\leq q^{-(m-1)\lct_Z(f)+\sigma_{I}}m^{\#(I)-1},
\end{equation}
where
\begin{equation}\label{eq:sigmaI1}
\sigma_I= -\sum_{i\in I} \big(\nu_i-N_{i}\lct_Z(f)\big).
\end{equation}
  \end{prop}

\begin{proof}
Since $h\big(E_I(\CC)\big)\cap Z(\CC)$ is non-empty, it follows that $\frac{\nu_i}{N_i} \geq \lct_Z(f)$ for all $i\in I$.
We have
$$
\text{Coeff}_{t^{m-1}}\prod_{i\in I}\frac{t^{N_{i}}q^{-\nu_{i}}}{1-t^{N_{i}}q^{-\nu_{i}}}=\sum_{(a_{i})_{i\in I}\in A_{I,m}}q^{-\sum_{i\in I}\nu_{i}(a_{i}+1)},
$$
where
$$
A_{I,m}=\{(a_i)_{i\in I}\in \NN^{\#I}\mid \sum_{i\in I}N_{i}(a_{i}+1)=m-1\}.
$$
For each $(a_{i})_{i\in I}\in A_{I,m}$, we have
\begin{align*}
-\sum_{i\in I}\nu_{i}(a_{i}+1) 
&=-(m-1)\lct_Z(f)-\sum_{i\in I}(a_{i}+1)(\nu_i-N_i\lct_Z(f))\\
&\leq -(m-1)\lct_Z(f)  -\sum_{i\in I}(\nu_i- N_{i}\lct_Z(f)),
\end{align*}
where the inequality follows from the fact that $\frac{\nu_i}{N_i}\geq \lct_Z(f)$.
Since $\#(A_{I,m})\leq m^{\#(I)-1}$, we obtain the assertion in the proposition.
\end{proof}

We can now prove our main results on exponential sums.

\begin{proof}[Proof of Theorem \ref{thm:sum}]
Let $f$, $\cO$, and $Z$ be as in Theorem \ref{thm:sum}. 
Let $V_f = \{z_1, \ldots, z_e\}$ be the set of critical values of $f$ over $\CC$.
Note that the $z_i$ are algebraic over $\QQ$, hence we can choose a nonzero integer $N$ such that the $Nz_i$ lie in the integral closure of $\ZZ$ inside $\CC$. Write $\cO_{z_i}$ for the integral closure of $\cO[Nz_i]$ inside its fraction field. Let $Z_i$ be the intersection of $Z$ with the closed subvariety of $\AA^n_{\cO_{z_i}}$ given by $Nf(x)=Nz_i$.
By \cite[Remark~4.5.3]{DenefBour} and  \cite[Lemma~5.1]{Saskia-Kien}, there is $M$ such that for any local field $L$ over $\cO$  of residue field characteristic  at least $M$ and
any nontrivial additive character $\psi$ on $L$ with $m_\psi\geq 2$,  we have
\begin{equation}\label{summand}
E^Z_{f,L,\psi} = \sum_{i}  E^{Z_i}_{f,L,\psi},
\end{equation}
where the sum is over those $i$ such that $L$ admits a unit preserving ring homomorphism  $\cO_{z_i}\to L$. (Here we use the fact  that $m_\psi\geq 2$.)
Up to working with each $Z_i$ separately, with  $f$ replaced by $N(f-z_i)$, it follows from (\ref{summand})  that it is enough to only consider $Z$ such that $f$ vanishes on $Z(\CC)$.
But this case follows by combining the estimates and equalities from Propositions \ref{004}, \ref{005}, \ref{trivial}, \ref{estimation1}, and Corollary \ref{cor:num}.
Indeed, there clearly exists a uniform bound on the number of characters of $\cO_L^{\times}$ \change{with order} dividing
$\prod_{i\in T_{K}} N_i$, and, by Propositions \ref{004} and \ref{005},  these characters are the only ones that can contribute to the exponential sums $E^Z_{f,L,\psi}$.
\end{proof}

By the uniformity in $Z$ in the results used in the proof of Theorem \ref{thm:sum}, we can show a uniform variant, where both constants $C$ and $M$ can be taken independently from $Z$. For simplicity of notation we focus on the case that $f$ vanishes on $Z(\CC)$, leaving
to the reader  the task of formulating and proving the variant of Theorem \ref{CVlct:unif} without the condition $f\big(Z(\CC)\big)=0$.
Consider a ring of integers $\cO_1$ containing $\cO$, and a closed subscheme $Z$ of $\AA^n_{\cO_1}$ with $f\big(Z(\CC)\big)=0$.
For such data and any finite field $\FF_q$ of large enough characteristic (depending only on $f$ and $h$) allowing a ring morphsim $\cO_1\to \FF_q$, define
$$
\tau_{Z,\FF_q}(f) = \min_{i}\frac{\nu_i}{N_i}
$$
where the minimum is taken over  those $i\in T_K$ such that $h(E_i)(\FF_q)$ has non-empty intersection with $Z(\FF_q)$, and using notation for the reduction of $h(E_i)$ as before.

\begin{thm}[Uniformity in $Z$]\label{CVlct:unif}
Let $N>0$ and $n>0$ be integers and let $f\in\cO[x_1,\ldots,x_n]$ be a non-constant polynomial over a ring of integers $\cO$.
Then there exist $C>0$ and $M>0$ such that for each ring of integers $\cO_1$ containing $\cO$, for each closed subscheme $Z$ of $\AA^n_{\cO_1}$ with $f\big(Z(\CC)\big)=0$ and such that $Z$ is defined by at most $N$ equations, each of degree at most $N$, for each local field $L$ over $\cO_1$ with residue field characteristic at least $M$, and for each nontrivial additive character $\psi$ on $L$ with $m_\psi\geq 2$, we have
 \begin{equation}\label{eq:unif1}
|E^{Z}_{f,L,\psi}| <  C m_{\psi}^{n-1} q_L^{ - \tau_{Z,k_L}(f) m_{\psi} }.
\end{equation}
\end{thm}

The case of Theorem \ref{CVlct:unif} with the extra assumption that $Z$ is a point in $\cO^n$ was predicted in \cite[Conjecture 1.2 (1.2.2)]{CVeys}.

\begin{proof}[Proof of Theorem \ref{CVlct:unif}]
The proof follows closely the proof of Theorem \ref{thm:sum}, by exploiting in addition
the uniformity assertions  in Propositions \ref{004}, \ref{005}, \ref{trivial} and \ref{estimation1}.
\end{proof}

Note that, given any $\cO_1$ and $Z$ as in Theorem \ref{CVlct:unif}, there is $M_Z$ such that $\tau_{Z,k_L}(f)$ equals  $ \sigma_Z(f)$
for all $L$ over $\cO_1$ with $p_L>M_Z$.  However, in Theorem \ref{CVlct:unif} we can take  $M$ independent of  $Z$.

 \subsection{Poles of largest possible order}\label{subsec:poles}
We formulate a consequence of our work to poles of maximal possible order for Igusa's local zeta functions, in the twisted case.
Recall that Veys' 1999 conjecture from \cite{VeysLaer}, solved by Nicaise and Xu in \cite{NiXu},
says that any pole of maximal possible order for Igusa's local zeta function associated to $f$ is of the form $-1/N$, for a positive integer $N$; moreover, in this case
the log canonical threshold of $f$ is equal to $1/N$. Recall that
for any polynomial $f$ in $n$ variables over $\cO$, and closed subscheme $Z$ of $\AA_\cO^n$, any local field $L$ over $\cO$ of characteristic zero and any character $\chi$ of $\cO_L^\times$, the maximal possible order of any pole of $\cZ_{f,L,\chi,s}^{Z}$  is $n$.
Nicaise and Xu treat the non-twisted case of Igusa's local zeta function with $Z=\{0\}$, namely, $\cZ_{f,L,\chi_{\rm triv},s}^{0}$, with $\chi_{\rm triv}$ being the trivial character.

Our work has consequences for poles of maximal possible order in the twisted case of Igusa's local zeta function around $Z$, namely, for $\cZ_{f,L,\chi,s}^{Z}$ with nontrivial $\chi$.
It is easy to check that if $s_0$ is a pole of maximal possible order of
$$
\cZ_{f,L,\chi_{\rm triv},s}^{Z},
$$
and if $\chi$ has order $d$, then $s_0/d$ is a pole of maximal possible order of
$$
\cZ_{f^d,L,\chi,s}^{Z}.
$$
Motivated by this observation and by \cite{VeysLaer}, \cite{NiXu}, one may wonder whether $s_0$ being a pole of maximal possible order of $\cZ_{f,L,\chi,s}^{Z}$ with $\chi$ a character of order $d> 1$ implies that
$s_0=-\lct_{Z}(f)=-\dfrac{1}{dk}$ for a positive integer $k$.
We obtain the following result in this direction.
\begin{prop}\label{prop:veys} Let $f$ and $Z$ be as above and let $d\geq 1$ be an integer. If $s_0=-\lct_Z(f)$ is a pole of maximal possible order of $\cZ_{f,L,\chi,s}^{Z}$ for infinitely many $L$ with arbitrarily large residue field characteristic and with $\chi$ a character of order $d$ on $\cO_L^\times$, then
$$
\lct_Z(f)\leq \dfrac{1}{d}.
$$
\end{prop}
\begin{proof} If $s_0=-\lct_{Z}(f)$ is a pole of maximal possible \change{order} of $\cZ_{f,L,\chi,s}^{Z}$ for infinitely many $L$ with arbitrarily large residue field characteristic, then the power condition holds for $(f,h,Z)$ witnessed by some $(I,W,g,d)$ with $|I|=n$, $d|N_i$ and $\lct_Z(f)=\frac{\nu_i}{N_i}$ for all $i\in I$. Indeed, this follows from Proposition \ref{005}. The proposition now follows by applying the bound (\ref{eq:Sub:weak:conj}) from Corollary \ref{cor:num} for this $I$ and $d$.
\end{proof}

\subsection{Optimality of the bounds and the motivic oscillation index}\label{sec:opt}

In this last section we give lower bounds for $|E_{f,L,\psi}^{Z}|$,
showing the optimality of the exponent $\sigma_Z(f)$ in the bounds of Theorem \ref{thm:sum} when $\sigma_Z(f)<1$ and in some cases also when $\sigma_Z(f)=1$.
In fact, we refine the notion of motivic oscillation index of $f$ from \cite{Cigumodp} to a variant $\moi_Z(f)$ around $Z$ (with a sign change compared to \cite{Cigumodp}), and show on one hand the optimality of bounds with $\moi_Z(f)$ in the exponent, and on the other hand, the equality $\moi_Z(f)=\sigma_Z(f)$ in the case of non-rational singularities. We will conclude by rephrasing the remaining part of Igusa's conjecture optimally in terms of $\moi_Z(f)$.

We use the notation in \S\ref{sec:Lchi}. Thus, $f$ is a non-constant polynomial in $n$ variables with coefficients in $\cO$, $Z$ is a closed subscheme of $\AA_\cO^n$, and $h$ is a log resolution of $D_K$, with $K$ the field of fractions of $\cO$. Also, in this section we only work with local fields $L$ which are either of characteristic zero, or of any characteristic, but with $p_L$ sufficiently large (depending on $f$ and $h$).

With the notation from Section \ref{sec:Lchi} for $\chi$, recall that $\cZ^Z_{f,L,\chi,s}$ equals a rational function $R(t)$ in $t=q_L^{-s}$, and denote by
$$
{\rm NP}^Z_{f,L,\chi}
$$
the set of complex numbers $s_0$ such that $t_0:= q_L^{-s_0}$ is a pole of
$$
(t-q_L)^\delta R(t)
$$
with $\delta=1$ if $\chi$ is trivial and with $\delta=0$ if $\chi$ is nontrivial. This is the \emph{set of nontrivial poles} of $\cZ^Z_{f,L,\chi,s}$.

Define
$$
{\rm LNP}^Z_{f,L,\chi} :=  \sup \{ \Re (  r  )\mid r \in {\rm NP}^Z_{f,L,\chi}\} ,
$$
where $\Re(r)$ stands for the real \change{part} of $r$,
and where the supremum over the empty set is taken to be $-\infty$. This is the (real part of the) \emph{largest nontrivial pole} of $\cZ^Z_{f,L,\chi,s}$.

If $f$ vanishes on $Z(\CC)$, then we define the motivic oscillation index of $f$ along $Z$ as
$$
\moi_Z(f) :=  -   \lim_{M\to+\infty} \sup_{L,\ p_L>M} \sup_{\chi}   \big( {\rm LNP}^Z_{f,L,\chi} \big),
$$
where $L$ runs over the local fields  \change{over any ring of integers containing the coefficients of $f$ and over which $Z$ can be defined,} with residue field characteristic $p_L$ larger than $M$ and $\chi$
runs over all multiplicative characters
$\cO_L^\times\to\CC^\times$. \change{Note that the limit of the suprema in the definition of $\moi_Z(f)$ stabilizes, see \cite[Corollary 3.4]{Cigumodp}. A definition of  $\moi_Z(f)$ for $f$ and $Z$ defined over the algebraic closure of $\Q$ rather than over a ring of integers is given in \cite{CMu}, where also some new inequalities are shown.}

For general $Z$, for any critical value \change{$z_i$ of $f$,} let $Z_i$ be the intersection of $Z$ with the closed subvariety of $\AA^n_{\cO_{z_i}}$ given by $Nf(x)=Nz_i$ for some nonzero integer $N$ such that the $Nz_i$ lie in the integral closure of $\ZZ$ inside $\CC$ and with $\cO_{z_i}$ as in the proof of Theorem \ref{thm:sum}. We define $\moi_Z(f)$ as the minimum over $i$ of the values $\moi_{Z_i}\big(N(f-z_i)\big)$.

It is a result of Igusa (see \cite[Theorem~2]{Igusa2}) that ${\rm LNP}^Z_{f,L,\chi}$ is a negative rational number or $-\infty$. Moreover, if
 $f\big(Z(\CC)\big)=0$, then $\moi_Z(f)$ is either $+\infty$, or it is a positive rational number equal to $\nu_i/N_i$ for some $i\in J$,
 where we use the notation in \S\ref{sec:resolK}. 

The following result gives lower bounds for the exponential sums in terms of $\moi_Z(f)$ in the exponent. It follows directly from
\cite[Theorem~2]{Igusa2} and the observation at the start of \cite[\S4]{Igusa2}. In combination with Proposition \ref{prop:moi:1} below, it shows \change{optimality} of our bounds in the case of non-rational singularities around $Z$. Indeed, \change{when $f(Z(\CC))=0$}, Proposition \ref{prop:moi:1} shows that $\lct_Z(f) = \moi_Z(f)$ if and only if $f$ has non-rational singularities on every open neighborhood $V$ of $Z_\CC$ in $\AA^n_\CC$.

\begin{prop}[\cite{Igusa2}]\label{prop:opt}
Given $f$ and $Z$ as above, there exist infinitely many local fields $L$ over $\cO$ (with arbitrarily large residue field characteristic), constants $c_L>0$, and positive integers $a$, $c$, such that
\begin{equation}\label{eq:lower}
c_L q_L^{ - \moi_Z(f) m} < |E^{Z}_{f,L,\psi}|
\end{equation}
for each $m\in c+a\NN$ and some additive character $\psi$  on $L$ with $m_\psi=m$.
\end{prop}
\begin{proof}
The assertion follows from \cite[Theorem~2]{Igusa2} (see also \cite[Corollary~1.4.5]{DenefBour} and the comment to \cite{DenefBour} at the end of \cite{DenefVeys}).
\end{proof}

\begin{prop}\label{prop:moi:1}
Given $f$ and $Z$ as above, there exists an open neighborhood $V$ of $Z_\CC$ in $\AA^n_\CC$ such that the following equivalences hold. All hypersurfaces defined by $f-b$,  for $b\in\CC$, have rational singularities on $V$ if and only if  $\moi_Z(f) >1$.
\change{If furthermore $f$ vanishes on $Z(\CC)$, then} we have $\moi_Z(f) \leq 1$ if and only if $\moi_Z(f) = \lct_Z(f)$.
\end{prop}
\begin{proof}
The proof relies mainly on \cite[\S3]{Igusa2}, using the fact that
 the observation at the beginning of  \cite[\S4]{Igusa2} removes the condition that $0$ is the only critical value.
Clearly it is enough to consider the case when $f\big(Z(\CC)\big)=0$. In what follows, we use the notation
in \S\ref{sec:resolK}
for a log resolution $h$.
If $f$ has no singularities on some $V$ containing $Z$, then $\moi_Z(f)=+\infty$, since then the sets ${\rm NP}^Z_{f,L,\chi}$ are empty whenever $p_L$ is large.

Note that the hypersurface defined by $f$ has rational singularities on some $V$ containing $Z_{\CC}$ if and only if for each $i\in J$ with $h(E_i)\cap Z$ non-empty either  $\nu_i/N_i>1$, or
  $(N_i,\nu_i)=(1,1)$ and for every other $E_{i'}$ with $(N_{i'},\nu_{i'})=(1,1)$, we have $E_i\cap E_{i'}\cap h^{-1}(Z)=\emptyset$
  (indeed, the latter condition is known as the pair $(\AA^n_{\CC},D)$ having canonical singularities in a neighborhood of $Z_{\CC}$;
  this is equivalent with $D$ having rational singularities in a neighborhood of $Z_{\CC}$ by \cite[Theorems~7.9 and 11.1]{Kollar}).
If these properties on the numerical data hold, then \cite[Theorem~2]{Igusa2} implies that  ${\rm LNP}^Z_{f,L,\chi}<-1$ for all $L$ and all $\chi$ whenever $p_L$ is large, and hence that $\moi_Z(f) >1$.

Conversely, suppose that $\moi_Z(f) >1$. \change{We need} to find an open neighborhood $V$ of $Z_{\CC}$ such that the hypersurface defined by $f$
has rational singularities in $V$.
 For this, we follow an argument already present in Igusa's work.
Since $\moi_Z(f) >1$, it follows from \cite[Theorem~2]{Igusa2} that for $p_L$ large, the function $L^\vee\to\CC$ sending a character $\psi$ in the (topological) dual $L^\vee$ of $(L,+)$ to the complex number $E^Z_{f,L,\psi}$ is an ${\rm L}^1$-function (with respect to the Haar measure).  By
 \cite[Theorem~2]{Igusa2}, for any $L$, this function $\psi\mapsto E^Z_{f,L,\psi}$  is ${\rm L}^1$ on $L^\vee$ if and only if the limit of
$$
F^Z_L(k):= \int_{ \{x\in \cO_{L}^n\mid \overline x\in Z(k_L),\ f(x)=k \} } |dx/df|
$$
exists for $k\to 0$, where $|dx/df|$ stands for the volume associated to the Gelfand-Leray differential form on $f(x)=k$ for smooth values $k\in L$ of $f$.
 By \cite[Lemma~4]{Igusa2}, we can find $V$ as needed if and only if for all $L$ with $p_L$ sufficiently large, and all real valued, non-negative Schwartz-Bruhat functions $\Phi$ on $L^n$ such that $f(\Supp \Phi)$ contains no critical value of $f$ other than $0$, where $\Supp \Phi$ is the support of $\Phi$, we have that the limit of
$$
F_\Phi(k):= \int_{ \{x\in L^n\mid \ f(x)=k \} }  \Phi(x) |dx/df|
$$
for $k\to 0$ exists. Note that $F_\Phi$ is non-negatively real valued since $\Phi$ is.
Moreover, it follows from \cite[Theorem~2]{Igusa2} that if the limit of $F_\Phi$ for $k\to 0$ does not exist, then for each $C$ there is $k\in L^\times$ with
$$
C < F_\Phi(k)
$$
\change{(that is, arbitrarily large values for $F_\Phi$ occur in this situation)}.
Note also that if  $f(\Supp \Phi)$ contains no critical value of $f$, then $F_\Phi$ is continuous on $L$. From this discussion and the additivity of $F_\Phi$ in non-negative real valued $\Phi$, it follows that we can find $V$ as desired, namely, an open neighborhood
of
 $Z_\CC$ in $\AA^n_\CC$ such that the hypersurface defined by $f$ has rational singularities on $V$. (Alternatively, one can use \cite[Theorem~3.4]{AizAvni} and the equivalence between its statements a and c to shorten the above argument.)

Let us now prove that $\moi_Z(f) \leq 1$ if and only if $\moi_Z(f) = \lct_Z(f)$. This follows again by an argument already present in Igusa's work.
The fact that if $\moi_Z(f) = \lct_Z(f)$, then $\moi_Z(f) \leq 1$ is clear, since we always have $\lct_Z(f) \leq 1$. Suppose now that $\moi_Z(f) \leq 1$.  As soon as $L$ contains a large enough finite field extension of the field of fractions $K$ of $\cO$, we have that $s=-\lct_Z(f)$ is a pole of
$\cZ^Z_{f,L,\chi_{\rm triv},s}$, by the proof of
\cite[Lemma~4]{Igusa2} (alternatively, and with more details, by \cite[Theorem~2.7]{VeysZ}). We  always
have $\moi_Z(f) \geq \lct_Z(f)$, since $\moi_Z(f)$ is either $+\infty$ or $\nu_i/N_i$ for some $i\in J$. If  $\lct_Z(f)<1$, then we are
 done, since \change{then clearly $-\lct_Z(f)$} is also a pole of $(t-q_L)R(t)$, where $R(q_L^{-s})= \cZ^Z_{f,L,\chi_{\rm triv},s}$. On the other hand,
 if $\lct_Z(f)=1$, then we are done, since $\moi_Z(f) \leq 1$ and $\moi_Z(f) \geq \lct_Z(f)$.
\end{proof}

Note that $\moi_Z(f)$ always gives upper bounds with constants $c_L$ depending on $L$ by Igusa's work \cite{Igusa1} and \cite{Igusa2} (see also \cite[Corollary~1.4.5]{DenefBour} and \cite{DenefVeys}), as follows.

\begin{prop}[Igusa]\label{prop:moi:upper}
Given $f$ and $Z$ there exist $M$ and for each $L$ over $\cO$ with $p_L>M$ a constant $c_L>0$ such that for all nontrivial additive characters $\psi$ on $L$ with $m_\psi\geq 2$, we have
\begin{equation}\label{eq:upper:moi}
 |E^{Z}_{f,L,\psi}| < c_L q_L^{ - \moi_Z(f) m_\psi} m_\psi^{n-1}.
\end{equation}
\end{prop}
Note that the exponent $n-1$ of $m_\psi$ in (\ref{eq:upper:moi}) is not always optimal and is related to the order of the largest nontrivial pole of Igusa's zeta functions; one can define naturally a multiplicity of the motivic oscillation index capturing the optimal exponent of $m_\psi$. Let us finally recall the strong form of Igusa's conjecture with the motivic oscillation index around $Z$, predicting that the constants $c_L$ in (\ref{eq:upper:moi}) can be taken independently from $L$ as soon as $p_L$ is large enough.  By the work of this paper, this remains open in general only if $\moi_Z(f)>1$.

\change{Let us finally note that in the last chapter of Nguyen's PhD thesis \cite{Kien:thesis}
it was shown that the inequality from (\ref{eq_thm_Z_1/2}) and Theorem~\ref{thm:sum} are equivalent 
 (without proving either in general).}   

\bibliographystyle{amsplain}
\bibliography{anbib}
\end{document}